\documentclass{article}
\usepackage{amsmath}
\usepackage{amsfonts}
\usepackage{amsthm}
\usepackage{amssymb}
\usepackage{multicol}
\usepackage{textcomp}
\usepackage{graphicx}
\usepackage{enumerate}
\usepackage{subfigure}
\newtheorem{Theorem}{Theorem}
\newtheorem{Lemma}{Lemma}
\newtheorem{Proposition}{Proposition}
\newtheorem{Definition}{Definition}
\newtheorem{Corollary}{Corollary}

\theoremstyle{remark}
\newtheorem{Remark}{Remark}

\theoremstyle{remark}

\theoremstyle{plain}

\theoremstyle{remark}

\newcommand{\R}{{\mathbb{R}}}

\newcommand{\Z}{{\mathbb{Z}}}

\newcommand {\N}{{\mathbb{N}}}

\newcounter{fig}

\begin{document}
\title{Vertex Isoperimetric Inequalities for a Family of Graphs on $\Z^k$}

\author{Ellen Veomett \thanks{The first author was supported for eight weeks during the summer of 2010 through the University of
Nebraska-LincolnÕs Mentoring through Critical Transition Points grant (DMS-0838463) from the National Science Foundation.}\\
\small Saint Mary's College of California \\
\and
A.J. Radcliffe \\
\small University of Nebraska-Lincoln}
%\date{July, 2011}

\maketitle

\begin{abstract}

We consider the family of graphs whose vertex set is $\Z^k$ where two vertices are connected by an edge when their $\ell_\infty$-distance is 1.  We prove the optimal vertex isoperimetric inequality for this family of graphs.  That is, given a positive integer $n$, we find a set $A\subset \Z^k$ of size $n$ such that the number of vertices who share an edge with some vertex in $A$ is minimized.  These sets of minimal boundary  are nested, and the proof uses the technique of compression.

We also show a method of calculating the vertex boundary for certain subsets in  this family of graphs.  This calculation and the isoperimetric inequality allow us to indirectly find the sets which minimize the function calculating the boundary.
\end{abstract}

\section{Introduction and Results}

For a metric space $(X,d)$ with a notion of volume and boundary, an isoperimetric inequality gives a lower bound on the boundary of a set of fixed volume.  Ideally, for any fixed volume, it produces a set of that volume with minimal boundary.    The most well-known  isoperimetric inequality states that, in Euclidean space, the unique set of fixed volume with minimal boundary is the Euclidean ball.   

More recently, questions of isoperimetric inequalities have arisen within various kinds of discrete spaces  \cite{MR0200192}, \cite{DiscTor}, \cite{MR1137765}, \cite{MR1909858}, \cite{MR1082843},  \cite{MR0168489}, \cite{MR1612869}.  Not only are questions regarding isoperimetric inequalities natural geometric questions, but they are known to have many implications in areas such as measure concentration  \cite{MR856576}, \cite{MR1849347}, \cite{MR1036275}, \cite{MR1899299} and the theory of random graphs \cite{MR2073175}, \cite{MR1141924}, \cite{MR1221553}, \cite{MR1872641}.

One of the more broadly known discrete isoperimetric inequalities is Harper's Theorem.  Harper's Theorem involves the discrete space $\{0,1\}^n$ and the $\ell_1$ metric; that is, the metric defined by the norm 
\begin{equation*}
||(x_1,x_2, \dots, x_n)||_1 = \sum_{i=1}^n |x_i|
\end{equation*}  
A graph $G = (V,E)$ is defined on the vertex set $V = \{0,1\}^n$ where the edge set $E$ consists of points whose distance in the $\ell_1$ metric is precisely 1:
\begin{equation*}
E = \{(u,v): u \in \{0,1\}^n, v \in \{0,1\}^n, ||u-v||_1 = 1\}
\end{equation*}
Here, the vertex boundary of a set $A \subset \{0,1\}^n$ consists of all points whose distance from $A$ in the graph metric is no more than 1:
\begin{equation*}
\partial A = \{v \in V: d(v,A) \leq 1\}.
\end{equation*} 
As usual, the distance between a point $x$ and a set $A$ is defined
\begin{equation*}
d(x,A) = \inf_{a \in A} d(x,a) = \inf_{a \in A}\{\text{the length of the shortest path from } x \text{ to } a\}
\end{equation*}
In other words, the vertex boundary consists of both $A$ and all of the neighbors of $A$.     In  \cite{MR0200192} Harper shows that one can define an ordering on the 0-1 cube $\{0,1\}^n$ such that a set of smallest vertex boundary is achieved on an initial segment.  Thus, the sets achieving the minimal boundary are nested.

In \cite{DiscTor}, Bollob\'as and Leader found an isoperimetric inequality for a graph which is also defined using the $\ell_1$ metric.  The vertices of this graph are the vertices of the discrete torus $\Z_m^n$, where $\Z_m$ denotes the integers modulo $m$ and $m$ is necessarily an even integer.  The edges of this graph are points whose distance using the $\ell_1$ metric is precisely 1.  In \cite{DiscTor}, Bollob\'as and Leader use a tool called a fractional system to show that the sets of minimal boundary of size $|\{x \in \Z_m^n: ||x-\vec{0}||_1 \leq r\}|$ ($r = 0, 1, 2, \dots$) are precisely those balls: $\{x \in \Z_m^n: ||x-\vec{0}||_1 \leq r\}$.  Thus, the sets of minimal vertex boundary of size $|\{x \in \Z_m^n: ||x-\vec{0}||_1 \leq r\}|$ ($r = 0, 1, 2, \dots$)  again are nested.

The fact that sets achieving optimality are nested is crucial in both  \cite{MR0200192} and \cite{DiscTor}, as they use the technique of compression.  Compression has been utilized in these and many other discrete isoperimetric problems to inductively find sets of minimal boundary, and relies heavily on the fact that the graph and its lower dimensional counterparts have sets of minimal boundary that are nested.    Discussions of compression as a technique in discrete isoperimetric problems can be found in \cite{MR1082842}, \cite{MR2035509},   \cite{MR1444247},  and \cite{MR1455181}. 

In the following we consider the vertex  isoperimetric inequalities of a family of graphs which was previously unstudied.  Specifically, we consider the family of graphs whose vertex set is $\Z^k$ where two vertices are connected by an edge when their $\ell_\infty$-distance is 1.   For each $n \in \N$, we produce a set in $\Z^k$ of size $n$ whose vertex boundary is minimized.  We do this by defining a well-ordering $\prec$ on $\Z^k$ and showing that a set of minimal boundary is an initial segment:

\begin{Theorem}\label{MyTheorem}
Let $I$ be an initial segment in $\Z^k$, and $A$ a finite, nonempty subset.  If $|I| = |A|$ then $|\partial I| \leq |\partial A|$.
\end{Theorem}

Thus, there exist sets of minimal boundary which are nested.  (The sets of minimal boundary are not unique, as we point out at the end of Section \ref{SectionMyTheorem}).  To prove Theorem \ref{MyTheorem}, we use a version of compression which is similar to that used to prove Harper's Theorem in \cite{MR1141923}.  Compression alone is not enough to produce a set of minimal boundary; instead we show that compressing a set and a particular type of ``jostling'' eventually results in a set of minimal boundary.  We prove Theorem \ref{MyTheorem} in Section \ref{SectionMyTheorem}.

In Section \ref{Boundary Computations}, we show how one can use a ``1-dimensional compression'' technique to take a set $A \subset \Z^k$ and produce a set of the same size whose boundary is no larger.  We call such a set centralized.  Then we are able to compute the boundary of centralized sets.  This computation along with the isoperimetric inequality of Theorem \ref{MyTheorem} allows us to deduce which sets achieve the minimum value of the function which computes the boundary of centralized sets.  This boundary computation technique may be useful in finding isoperimetric inequalities for other related graphs. 

In Section \ref{Final Remarks} we make some final comments on using the ideas in Section \ref{Boundary Computations} for other graphs.

\section{The Proof of Theorem \ref{MyTheorem}}\label{SectionMyTheorem}

We first consider the graph with vertex set $\Z^k$.   Two vertices $x, y \in \Z^k$ are joined by an edge precisely when $\|x-y\|_\infty = 1$.   That is, when
\begin{equation*}
\max_{i = 1, 2, \dots, k} |x_i-y_i| = 1
\end{equation*}
where $x = (x_1, x_2, \dots, x_k)$ and $y = (y_1, y_2, \dots, y_k)$.

We define a well-ordering $\prec$ on $\Z^k$ inductively.  For $\Z^1$, the well-ordering $\prec$ is:
\begin{equation*}
0 \prec 1 \prec -1 \prec 2 \prec -2 \prec 3 \cdots
\end{equation*}

For $k>1$, the well-ordering $\prec$ on $\Z^k$ is as follows: for $u = (u_1, u_2, \dots, u_k) \in \Z^k$ define
\begin{equation*}
M(u) = \max_\prec \{u_i: i=1, 2, \dots, k\}
\end{equation*}
where the maximum is according to the previously defined well-ordering on $\Z$.  Then for $u, v \in \Z^k$ with $u \not= v$, if $M(u) \prec M(v)$, then $u \prec v$.  If $M(u) = M(v)$, define
\begin{align*}
i_u &= \min\{i: u_i = M(u) \text{ where } u = (u_1, u_2, \dots u_k)\} \\
i_v &=   \min\{i: v_i = M(v) \text{ where } v = (v_1, v_2, \dots v_k)\}
\end{align*}
If $i_v<i_u$ then $u \prec v$.  Finally, if $M(u) = M(v)$ and  $i_u = i_v$, we define
\begin{align*}
u'= (u_1, u_2, \dots, u_{i_u-1}, u_{i_u+1}, \dots u_k) \in \Z^{k-1} \\
v'= (v_1, v_2, \dots, v_{i_u-1}, v_{i_u+1}, \dots v_k) \in \Z^{k-1} 
\end{align*}
and state that $u\prec v$ precisely when $u' \prec v'$.

For example, for $n = 3$, here are the first forty elements according to this well-ordering:

\begin{equation*}
\begin{array}{lllll}
1. (0,0,0) & 2. (0,0,1) & 3. (0,1,0) & 4. (0,1,1) & 5. (1,0,0) \\
6. (1,0,1) & 7. (1,1,0) & 8. (1,1,1) & 9. (0,0,-1)& 10. (0,1,-1)  \\
11. (1,0,-1) & 12. (1,1,-1) & 13. (0,-1,0) & 14. (0, -1, 1) & 15. (1,-1,0) \\
16. (1,-1,1) & 17. (0,-1,-1) & 18. (1,-1,-1) & 19. (-1, 0,0) & 20.  (-1,0,1) \\
21. (-1,1,0) & 22. (-1, 1, 1) & 23. (-1,0,-1)& 24. (-1,1,-1) & 25. (-1,-1,0)  \\
26. (-1,-1,1) & 27. (-1,-1,-1) & 28. (0,0,2) & 29. (0,1,2) & 30. (1,0,2) \\
31. (1,1,2) & 32. (0,-1,2) & 33. (1,-1,2) & 34. (-1,0,2) & 35. (-1,1,2) \\
36. (-1,-1,2) & 37. (0,2,0) & 38. (0,2,1) & 39. (1,2,0) & 40. (1,2,1)
\end{array}
\end{equation*}

We use the notation $u \preceq v$ to mean that either $u = v$ or $u \prec v$.   For $a \in \Z$, we let $a^s$ denote its immediate successor in the well-ordering $\prec$.  Thus, for example, 
\begin{align*}
1^s &= -1 \\
-1^s &= 2 \\
0^s &= 1
\end{align*}

Since it will be used later, we present the following remark on how to calculate immediate successors using this well-ordering $\prec$ on $\Z^k$:

\begin{Remark}[Calculating Successors]\label{Successors} \hfill

Consider $x = (x_1, x_2, \dotsc, x_k) \in \Z^k$.  Let $m$ be the smallest entry of $x$ according to the well ordering $\prec$ of $\Z$.  Thus, 
\begin{equation*}
m = x_i \text{ for some } i
\end{equation*}
and 
\begin{equation*}
m \preceq x_i \text{ for each } i
\end{equation*}
Define
\begin{equation*}
i_m = \max\{i: x_i = m\}.
\end{equation*}
If $i_m = k$, then $x_1 = x_2 = \dotsc = x_k$ and the immediate successor of $x$ is $(0, 0, \dotsc, 0, x_k^s)$.

Otherwise, $i_m<k$.  In this case, the immediate successor to $x$ is the vector $x^s = (y_1, y_2, \dots, y_k)$ where $x^s$ has the same entries as $x$ except that all entries $x_j$ which are equal to $x_{i_m}$ with $j<i_m$ change to 0 (or stay 0 if $x_{i_m} = 0$), $x_{i_m}$ changes to $x_{i_m}^s$, and all entries $x_j$ which are equal to $x_{i_m}^s$ with $j>i_m$ change to 0.  That is,
\begin{equation*}
y_i = \begin{cases}
0 & \text{ if } i< i_m \text{ and } x_i = x_{i_m} \text{ or if } i> i_m \text{ and } x_i = x_{i_m}^s \\
x_{i_m}^s & \text{ if } i= i_m \\
x_i  & \text{ otherwise }
\end{cases}
\end{equation*}

\end{Remark}

  For a set $A \subset \Z^k $, we use the notation $\partial A$ to denote 
 \begin{equation*}
 \partial A = \{x \in \Z^k: \|x-a\|_\infty \leq 1 \text{ for some } a \in A\}.
 \end{equation*}
Thus, $A \subset \partial A$.  We note that this is the notation used in \cite{MR1141923}, \cite{MR1137765}, \cite{MR1082843}, and others.  We use $|A|$ to denote the size of a finite set.

Our goal will be to prove that, for any $A \subset \Z^k$ $|\partial A| \geq |\partial I_{|A|}|$ where $I_{|A|}$ is the initial segment in $\Z^k$ of length $|A|$, according to the well-ordering $\prec$.  In order to prove this, we will need some Lemmas about initial segments in $\Z^k$ and their boundaries.

For a number $a \in \Z$, we let $a^+$ denote the element of $\{a, a+1, a-1\}$ which is largest in the well-ordering $\prec$ and $a^-$ the element of $\{a, a+1, a-1\}$ which is the smallest.    Thus, for example, 
\begin{align*}
1^+ &= 2 & 1^- &= 0 \\
-1^+ &= -2 & -1^- &= 0 \\
0^+ &= -1 & 0^-  & = 0 \\
\end{align*}

\begin{Lemma}
If $I\subset \Z^k$ is an initial segment, so is $\partial I$
\end{Lemma}

\begin{proof}
By induction on $|I|$.  If $|I| = 1$, then $I = \{(0, 0, \dotsc, 0)\}$ and $\partial I = \{0, 1, -1\}^k$ which is the initial segment of size $3^k$.  Now let $v \in \Z^k$, and suppose that $I = \{x \in \Z^k: x \preceq v\}$ and $\partial I$ are initial segments.  Let $u$ be the successor of $v$.  Note that $u$ is not the zero vector.  We will show that $\partial \left( I \cup \{u\}\right)$ is an initial segment.  

Let $v = (v_1, v_2, \dotsc, v_k)$.  Note that the vector in $\Z^k$ which is an element of $\partial I$ and the latest in the well-ordering $\prec$ is $(v_1^+, v_2^+, \dots, v_k^+)$.  Thus, by our inductive assumption,
\begin{equation*}
\partial I = \{x \in \Z^k: x \preceq (v_1^+, v_2^+, \dots, v_k^+)\}.
\end{equation*}

Note that 
\begin{equation*}
\partial (I \cup \{u\}) = \partial I \cup \{(u_1 +\delta_1, u_2 +\delta_2, \dotsc, u_k+\delta_k): (\delta_1, \delta_2, \dots, \delta_k) \in \{0, 1, -1\}^k\}.
\end{equation*}

Let $i \in \{1, 2, \dots, k\}$ and suppose that $u_i\not=0$.  Then 
\begin{equation*}
(u_1, u_2, \dotsc, u_{i-1}, u_i^-, u_{i+1}, \dotsc, u_n) \prec (u_1, u_2, \dotsc, u_n)
\end{equation*}
and for $(\delta_1, \delta_2, \dots, \delta_{k-1}) \in \{0,1,-1\}^{k-1}$ each of 
\begin{align*}
&  (u_1 +\delta_1, u_2 + \delta_2 , \dotsc, u_{i-1}+ \delta_{i-1}, u_i^-, u_{i+1}+\delta_i, \dotsc, u_k+\delta_{k-1})\\
&  (u_1 +\delta_1, u_2 + \delta_2 , \dotsc, u_{i-1}+ \delta_{i-1}, u_i, u_{i+1}+\delta_i, \dotsc, u_k+\delta_{k-1})
\end{align*}
are adjacent to $(u_1, u_2, \dotsc, u_{i-1}, u_i^-, u_{i+1}, \dotsc, u_n)$.

Thus, the elements $(x_1, x_2, \dotsc, x_n)$ of $\partial (I \cup \{u\}) $ which are \emph{not} in $\partial I$ are precisely those of the form:
\begin{equation*}
x_i = \begin{cases} 
u_i^+ & \text{ if } u_i\not= 0 \\
0, 1, \text{ or } -1 & \text{ if } u_i = 0
\end{cases}
\end{equation*}
Note that if $u_i \not= 0$ then $u_i^+$ is at least as large as 2 in the well-ordering $\prec$ on $\Z$.  Finally, recall that $u$ is the immediate successor to $v = (v_1, v_2, \dots, v_k)$ and that every entry of $(v_1^+, v_2^+, \dots, v_k^+)$ is at least as large as $-1$ in the well-ordering on $\Z$.  Thus, by Remark \ref{Successors}, we can see that the elements of $\partial (I \cup \{u\}) $ which are \emph{not} in $\partial I$ are the $3^\ell$ immediate successors of $(v_1^+, v_2^+, \dots, v_k^+)$ where $\ell$ is the number of 0 entries of $u$.

\end{proof}

We note that we have also proved the following:

\begin{Lemma}\label{count}
Suppose $I$ is an initial segment in $\Z^k$ and let $v$ be the first element not in $I$.  Then
\begin{equation*}
|\partial (I \cup \{v\})| = |\partial I| +3^\ell 
\end{equation*}
where $\ell$ is the number of coordinates equal to 0 in $v$.
\end{Lemma}

The following technical Lemma will be used in the proof of Theorem \ref{MyTheorem}

\begin{Lemma}\label{fill}
Suppose $k \geq 2$ and consider a segment in $\Z^k$:
\begin{equation*}
S = \{v \in \Z^k: v_0 \leq v \leq v_1\}, \quad k \geq 2.
\end{equation*}
Let $v_0 = (v_{0,1}, v_{0,2}, \dots, v_{0,k})$ and suppose that $v_{0,j} = 0$ for all $j \not= j^*$ and $v_{0,j^*} \not= 0$.  Suppose that the last $t$ vectors in $S$ all do not have a 0 entry.  Then for any $j \in \{1, 2, \dots, k\}$, $j \not= j^*$ and for any
\begin{equation*}
s \in \left\{1,-1, 2, -2, \dots   (-1)^{t+1}\left\lceil  \frac{t}{2} \right\rceil \right\}
\end{equation*}
there is a vector in $S$ whose $j$th entry is $s$.
\end{Lemma}

\begin{proof}
Note that the numbers in the set $\left\{1,-1, 2, -2, \dots   (-1)^{t+1}\left\lceil  \frac{t}{2} \right\rceil \right\}$ are the first $t$ elements in the list $1, -1, 2, -2, 3, -3, \dots$.

Recall from Remark \ref{Successors} how immediate successors in $\Z^k$ are calculated.  The only way that $t$ successive vectors all do not have a 0 entry is if a single coordinate is increasing (according to $\prec$) in those successive vectors.  

Suppose it is the $i$th coordinate which is increasing in the final $t$ vectors of $S$.  This implies that, in \emph{each} of the $t$ successive vectors that have no 0 coordinate, the $i$th coordinate is smaller than any of the other coordinates (in the $\prec$ ordering on $\Z$).  Thus, for \emph{any} $\ell \not= i$, the $\ell$th coordinate must be at least as large as the $t$th element in the list $1, -1, 2, -2, \dots$.  Hence we can see that every coordinate of $v_1$ has a value at least as large as the $t$th element in the list $1, -1, 2, -2, 3, \dots$.  

Consider any $j \not=j^*$.   By assumption, the $j$th coordinate $v_0$ is 0.   Recall that every coordinate of $v_1$ has a value at least as large as the $t$th element in the list $1, -1, 2, -2, 3, \dots$ and $S$ is the segment between $v_0$ and $v_1$.  Thus, by the definition of the ordering $\prec$ on $\Z^k$, for the 1st through $t$th elements in the list $1, -1, 2, -2, \dots$,  there must be an element between $v_0$ and $v_1$ whose $j$th coordinate is equal to that element.

Specifically, let $v_0 = (v_{0,1}, v_{0,2}, \dots, v_{0,k})$ and $s \in \{1, -1, 2, \dots, (-1)^{t+1}\left \lceil \frac{t}{2} \right\rceil$\}.  Then the vector $x=(x_1, x_2, \dots, x_k)$ where
\begin{equation*}
x_\ell = \begin{cases}
v_{0,\ell} & \text{ if } \ell \not= j \\
s & \text{ if } \ell=j
\end{cases}
\end{equation*}
has an entry equal to $s$ and $v_0 \prec x \preceq v_1$.   Thus, we have proven the Lemma.

\end{proof}

The notation we use (and the technique used) in proving Theorem \ref{MyTheorem} is similar to that in that in the proof of Harper's Theorem in \cite{MR1141923}.  For $A \subset \Z^k$, $i \in \{1, 2, \dots, k\}$, and $j \in \Z$ we define $A_{i^j}$ to be the set of vectors in $\Z^{k-1}$ such that, when a $j$ is inserted in between the $i-1$th and $i$th entries, the resulting vector is in $A$.  That is,
\begin{equation*}
A_{i^j} = \{ x \in \Z^{k-1} : (x_1, x_2, \dotsc, x_{i-1}, j, x_i, x_{i+1}, \dotsc, x_{k-1}) \in A\}
\end{equation*}

%We write $e_1, e_2, \dots, e_k$ to denote the standard basis vectors in $\Z^k$.  For $S \subset \{1, 2, \dots, k\}$ define $\widehat{S}$ to be the complement of $S$ in $\{1, 2, \dots, k\}$:
%\begin{equation*}
%\widehat{S} = \{i \in \{1, 2, \dots, k\}: i \not\in S\}
%\end{equation*}
%Additionally, for $S \subset \{1, 2, \dots, k\}$, we define
%\begin{equation*}
%\Z^S = \{x \in \Z^k: x_i = 0 \text{ for } i \not\in S\}
%\end{equation*}
%For singleton sets, we suppress brackets so that
%\begin{equation*}
%\Z^{\widehat{i}} = \{x \in \Z^k: x_i = 0\}
%\end{equation*}

%Now, for $S \subset \{1, 2, \dots, k\}$, $A \subset \Z^k$, and $x \in \Z^{\widehat{S}}$, we denote by $A_S(x)$ the $S$-section of $A$ at $x$:
%\begin{equation*}
%A_S(x) = \{y \in \Z^S:x+y \in A\}
%\end{equation*}
%Thus, for $i \in \{1, 2, \dots, k\}$, $j \in \Z$, and $A \subset \Z^k$,
%\begin{multline*}
%A_{\widehat{i}}(je_i) 
%= \{(x_1,  \dots, x_{i-1}, 0, x_{i+1}, \dots, x_k) \in \Z^k \\
%: (x_1, \dots, x_{i-1}, j, x_{i+1}, \dots, x_k) \in A\}
%\end{multline*}

Now we are ready to prove Theorem \ref{MyTheorem}.  

\begin{proof}[Proof of Theorem \ref{MyTheorem}]
Let $A \subset \Z^k$ be finite and nonempty.  We proceed by induction on $k$.  For $k = 1$, one can easily see that 
\begin{equation*}
|\partial A| = |A|+2
\end{equation*}
if $A$ is a segment: $A = \{j \in \Z: a \leq j \leq b \text{ for some } a, b \in \Z\}$ and
\begin{equation*}
|\partial A| > |A|+2
\end{equation*}
if $A$ is not a segment.  Since every initial segment in $\Z$ according to the well-ordering $\prec$ is a segment, the Theorem is proved for $k = 1$.

Now suppose $A \subset \Z^k$ is finite and nonempty, $k>1$, and the theorem is proven for all smaller positive $k$.  Let $i \in \{1, 2, \dots, k\}$.  Note that there are only finitely many $j \in \Z$ for which the set
\begin{equation*}
A_{i^j}
\end{equation*}
is nonempty, and that 
\begin{equation*}
\sum_{j \in \Z} |A_{i^j}| = |A|
\end{equation*}

We define a set $C_i$ which is the ``$i$-compression'' of $A$ by specifying its $k-1$-dimensional sections fixing the $i$th coordinate.  Specifically, for $j \in \Z$,  we define $(C_i)_{i^j}$ to be the initial segment in $\Z^{k-1}$ of size $|A_{i^j}|$.  This gives $|(C_i)_{i^j}| = |A_{i^j}|$ for each $j$.  Thus, since
\begin{align*}
\sum_{j\in \Z} |A_{i^j}| &= |A| \\
\text{ and} & \\
\sum_{j\in \Z} |(C_i)_{i^j}| &= |C_i|
\end{align*}
we have $|C_i| = |A|$.  Now note that
\begin{align*}
(\partial A)_{i^j} &= \partial(A_{i^j}) \cup \partial (A_{i^{j-1}}) \cup \partial (A_{i^{j+1}}) \\
(\partial C_i)_{i^j} &= \partial ((C_i)_{i^j}) \cup \partial ((C_i)_{i^{j-1}}) \cup \partial ((C_i)_{i^{j+1}})
\end{align*}
By definition, we know that $|(C_i)_{i^j}| = |A_{i^j}|$, $|(C_i)_{i^{j-1}}| = |A_{i^{j-1}}|$, and $|(C_i)_{i^{j+1}}| = |A_{i^{j+1}}|$.  Also, $(C_i)_{i^{j}}$ is an initial segment, so by induction we have $|\partial ((C_i)_{i^j})| \leq | \partial(A_{i^j})|$.  Similarly for $j+1$ and $j-1$.   

If the union above were necessarily disjoint, we would be done.  Since it's not, we note the following: all of $\partial((C_i)_{i^j}) $, $\partial ((C_i)_{i^{j-1}})$ and $\partial ((C_i)_{i^{j+1}})$ are initial segments.  Thus, they are ordered by containment.  Thus, either the inequality $|\partial ((C_i)_{i^j})| \leq | \partial(A_{i^j})|$ or one of the equalities $|\partial ((C_i)_{i^{j-1}})| \leq |\partial (A_{i^{j-1}})|$, $|\partial ((C_i)_{i^{j+1}})| \leq |\partial(A_{i^{j+1}})|$ is enough to give us that $|(\partial C_i)_{i^j}| \leq |(\partial A)_{i^j}|$.  

Note that $j \in \Z$ was arbitrary.   Thus, since  
\begin{align*}
\sum_{j\in \Z} |(\partial A)_{i^j}| &= |\partial A| \\
\sum_{j\in \Z} |(\partial C_i)_{i^j}| &= |\partial C_i|
\end{align*}
we have $|\partial C_i| \leq  |\partial A|$.

Hence, we have shown that, for $i \in \{1, 2, \dots, k\}$, the $i$th-compression of $A$ is a set of the same size with no larger boundary.  We note that a set which is compressed in every coordinate need not be an initial segment.    For example, for $k = 3$, the set 
\begin{equation*}
\{(0, 0, 0), (0, 0, 1), (0, 1, 0), (1, 0, 0)\}
\end{equation*}
is $i$-compressed, but not an initial segment.  

Let $X \subset \Z^k$ be a set of size $|A|$ which has the following properties:
\begin{enumerate}
\item $|\partial X| \leq |\partial A|$
\item $X$ is compressed in every coordinate $i \in \{1, 2, \dots, k\}$.
\item If $Z$ is any other set of size $|A|$ satisfying the above properties 1 and 2, let $\ell(X)$ and $\ell(Z)$ be the last elements of $X$ and $Z$ in the well-ordering $\prec$ respectively.  Then we must have
\begin{equation*}
\ell(X) \preceq \ell(Z)
\end{equation*}
\end{enumerate}

In words: $X$ is a set of size $|A|$ with boundary no larger than the boundary of $A$ which is compressed in every coordinate (the previous arguments show that such an $X$ exists).  We also choose $X$ such that it is ``as close as possible to an initial segment'' as described in property 3: its last element is as small as possible.  We will show that $X$ must be an initial segment by contradiction.

Suppose that $X$ is not an initial segment.  Then there exists $x \not \in X$, $ y\in X$, such that $x\prec y$.  We note that, if $x_i = y_i$ for some $i \in \{1, 2, \dots, k\}$, then defining
\begin{align*}
x^* &= (x_1, \dots, x_{i-1}, x_{i+1}, \dots, x_k) \\
y^* &= (y_1, \dots, y_{i-1}, y_{i+1}, \dots, y_k)
\end{align*}
we have $x^*\prec y^*$ in $\Z^{k-1}$, contradicting the fact that $X$ is $i$-compressed.  Thus, each of the coordinates of $x$ and $y$ must be different.  

Let $S_\alpha$ be the first segment in $X$.  That is, $S_\alpha\subset X$, $S_\alpha$ is a segment according to the ordering $\prec$, and the immediate successor to the last element in $S_\alpha$ is not in $X$.   Note that, since $X$ is compressed in every coordinate, this implies that $(0, 0, \dots, 0) \in S_\alpha$.  Let $y$ be the first element in $X \backslash S_\alpha$, and let $x$ be its immediate predecessor.  From the above comments, we know that either 
\begin{equation*}
x = (0, 1, 1, \dotsc, 1) \quad y = (1, 0, 0, \dotsc, 0)
\end{equation*}
or $x_1 \not= 0$ and
\begin{equation*}
x = (\underbrace{x_1, x_1, \dotsc, x_1}_\ell,  \underbrace{x_1^s, x_1^s, \dotsc, x_1^s}_{n-\ell}) \quad y = (\underbrace{0, 0, \dotsc, 0}_{\ell-1}, x_1^s, \underbrace{0, 0, \dotsc, 0}_{k-\ell})
\end{equation*}

\underline{Case 1:}
\begin{equation*}
x = (0, 1, 1, \dotsc, 1) \quad y = (1, 0, 0, \dotsc, 0)
\end{equation*}
Let $S_\Omega$ be the last segment in $X$.  That is, $S_\Omega =\{w \in \Z^k: u \preceq w \preceq v\}$, $S_\Omega \subset X$, $v$ is the last element in $X$, and the immediate predecessor to $u$ is not in $X$.   Since $x \not \in X$ and $u \in X$,  by the above arguments we must have
$u_1 \not= 0$ and
\begin{equation*}
u = (u_1, \underbrace{0, 0, \dotsc, 0}_{k-1})
\end{equation*}
for some $u_1 \not= 0$.  Additionally, we note that the successor to $u$ is $(u_1, 0, \dots, 0, 1)$ which must have a coordinate in common with $x$.  Thus, by the above arguments, we must have $u = v$.  We note that by Lemma \ref{count}, there are at least $3^{k-1}$ elements which adjacent to $u$ and are not adjacent to any other element in $X$.  There are $3$ elements which are adjacent to $x$ and not adjacent to any element of $S_\alpha \subset X$.  Thus, we can see that 
\begin{equation*}
\partial((X \cup\{x\})\backslash\{u\}) \leq \partial(X)-3^{k-1}+3 \leq \partial(X), \quad |(X \cup\{x\})\backslash\{u\}| = |X|
\end{equation*}
and the set $(X \cup\{x\})\backslash\{u\}$ is closer to an initial segment with respect to property 3 above.  Thus, we have reached a contradiction in Case 1.

\underline{Case 2:}
$x_1 \not= 0$ and
\begin{equation*}
x = (\underbrace{x_1, x_1, \dotsc, x_1}_\ell,  \underbrace{x_1^s, x_1^s, \dotsc, x_1^s}_{k-\ell}) \quad y = (\underbrace{0, 0, \dotsc, 0}_{\ell-1}, x_1^s, \underbrace{0, 0, \dotsc, 0}_{k-\ell})
\end{equation*}
Let $S_\Omega$ be the last segment in $X$.  That is, $S_\Omega =\{w \in \Z^k: u \preceq w \preceq v\}$, $S_\Omega \subset X$, $v$ is the last element in $X$, and the immediate predecessor to $u$ is not in $X$.

Let $S_\beta$ be the first segment not in $X$.  That is, $S_\beta = \{z \in \Z^k: q \preceq z \preceq x\}$ and the immediate predecessor of $q$ is the last element of $S_\alpha$.  Note that, by the above arguments, since $q \not\in X$ and $y \in X$ all of the entries of $q$ and $y$ must be different.  Thus, the vector $q$ is somewhere in between 
\begin{equation*}
\begin{cases}
(\underbrace{x_1,  \dotsc, x_1}_{\ell-2}, 1, x_1, \underbrace{x_1^s, \dotsc, x_1^s}_{k-\ell})  \text{ and } 
x = (\underbrace{x_1,  \dotsc, x_1}_\ell,  \underbrace{x_1^s,  \dotsc, x_1^s}_{k-\ell}) & \text{if } 1<\ell<k \\
(x_1, x_1^s, \dots, x_1^s, 1) \text{ and } x = (x_1, x_1^s, x_1^s, \dots, x_1^s) &  \text{if }  \ell=1 \\
(x_1, \dots, x_1, 1 ) \text{ and } x = (x_1, x_1, \dots, x_1) &  \text{if }  \ell=k
\end{cases}
\end{equation*}
If $q$ has no 0 entries, then we know that there is no more than $3^0=1$ element in $\Z^k$ adjacent to $q$ which is not adjacent to any other element of $S_\alpha \subset X$.  Also there are $3^t \geq 1$ elements in $\Z^k$ larger than $v$ which are adjacent to $v$ and not adjacent to any other element in $X$, where $t$ is equal to the number of 0 entries of $v$.  Thus, we have
\begin{equation*}
\partial((X \cup \{q\}\})\backslash\{v\}) \leq \partial(X)-3^{t}+1 \leq \partial(X), \quad |(X \cup\{q\}\})\backslash\{v\}| = |X|
\end{equation*}
and $(X \cup \{q\}\})\backslash\{v\}$ is closer to an initial segment with respect to property 3 above.

It is possible for $q$ to have a 0 entry.  Specifically, we could have
\begin{equation*}
q= (\underbrace{x_1,  \dotsc, x_1}_{\ell-1}, 0, \underbrace{x_1^s, \dotsc, x_1^s}_{k-\ell})  \quad  \text{if } 1<\ell<k 
\end{equation*}
or we could have 
\begin{equation*}
q= (0, x_1^s, \dots, x_1^s)  \quad  \text{if }  \ell=1\end{equation*}
If $q$ does have a 0 entry, then we exchange for $q$ instead of $v$, we exchange for the immediate successor of $q$ instead of the immediate predecessor of $v$, and so on until either we have exchanged out all of $S_\Omega$ or filled all of $S_\beta$.  This creates a new $X'$ of the same size as $X$.  Note that $q$ is the only element of $S_\beta$ which has a 0 entry.  We claim that at least one of the elements of $S_\Omega$ that we exchanged for an element of $S_\beta$ had a 0 entry.  Indeed, if all of $S_\Omega$ is exchanged, we know that $u$ has a 0 entry.  Otherwise, only the last $t$ entries of $S_\Omega$ were exchanged, where $x_1$ is the $t$th entry in the list $0, 1, -1, 2, -2, \dots$.  If none of them had a 0 entry, then Lemma \ref{fill} would imply that one of the vectors in $S_\Omega$ has an entry equal to $x_1^s$ and another vector in $S_\Omega$ has that same entry equal to $x_1$.  This contradicts the fact that no element of $S_\Omega$ can share any entries with $x$.  

Thus, at least one of the elements of $S_\Omega$ that we exchanged for an element of $S_\beta$ had a 0 entry.  Thus, by Lemma  \ref{count}, we know that the boundary of the new set $X'$ is no larger than the boundary of $X$ and that $X'$ is closer to an initial segment with respect to property 3 above.  Thus, we have reached a contradiction in Case 2.  

Hence, $X$ must indeed be an initial segment, and we have proved our Theorem.

\end{proof}

We note that the sets that we found here of minimum boundary are not unique.  For example, the sets in Figure \ref{Fig1} are both sets in $\Z^2$ of size 10 of minimal boundary, but they are not isomorphic.  The blue dots represent the set, while the red dots represent the additional vertices which are in the boundary of the set.

\begin{figure}[h]
\begin{center}
\subfigure{\label{Example1}\includegraphics[width=1.4 in]{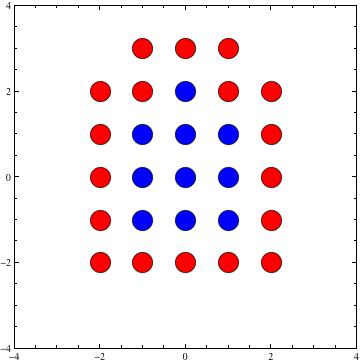}}
\hspace{.1 in} \subfigure{\label{Example2}\includegraphics[width=1.4 in]{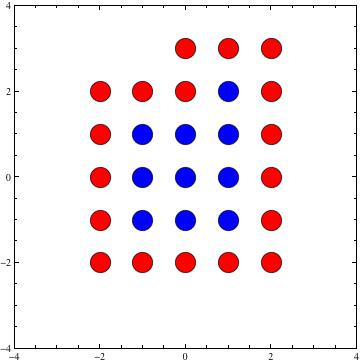}}
\end{center}
\label{Fig1}\caption{The set on the left is an initial segment, while the set on the right is not.}
\end{figure}

Finally, we note that the same arguments utilized above can also be used to prove a similar isoperimetric inequality for $\N^k$.  That is, we consider the graph whose vertex set is $\N^k$ such that there is an edge between $x\in \N^k$ and $y\in \N^k$ precisely when $\|x-y\|_\infty = 1$.  A well-ordering $\prec_\N$ is defined on $\N^k$ inductively just as it is for $\Z^k$, but with the base case of $\N^1$ being the standard ordering of $\N$:
\begin{equation*}
0<1<2<3< \cdots
\end{equation*}

All of the Lemmas and the Theorem are proven similarly.  In the case of $\N^k$, the equation of Lemma \ref{count} would instead conclude that
\begin{equation*}
|\partial (I \cup \{v\})| = |\partial I| +2^\ell 
\end{equation*}
where $\ell$ is the number of coordinates equal to 0 in $v$.  With this well-ordering $\prec_\N$, we have the following:
\begin{Corollary}\label{TheoremN}
Let $I$ be an initial segment in $\N^k$, and $A$ a finite, nonempty subset.  If $|I| = |A|$ then $|\partial I| \leq |\partial A|$.
\end{Corollary}

\section{Boundary Computations}\label{Boundary Computations}

We define a graph on the vertex set $\Z^k \times \N^d$, where $k$ and $d$ are non-negative integers (not both 0).  We use $\N$ to denote the set of non-negative integers:
\begin{equation*}
\N = \{x \in \Z: x \geq 0\}
\end{equation*}
As usual, we say that two vertices $x, y \in \Z^k \times \N^d$ are joined by an edge precisely when their $\ell_\infty$-distance is 1.  That is, when
\begin{equation*}
\|x-y\|_\infty = \max_{i=1, 2, \dots, k+d} \{|x_i-y_i|\}=1.
\end{equation*}

For simplicity, we introduce the following notation: for a real-valued vector $p = (p_1, p_2, \dots, p_n) \in \R^n$, and $x \in \R$, we define
\begin{equation*}
(p, x \rightarrow i) = (p_1, p_2, \dots, p_{i-1}, x, p_i, p_{i+1}, \dots, p_n) \in \R^{n+1}
\end{equation*}
In words, $(p, x \rightarrow i)$ is the vector that results when placing $x$ in the $i$th coordinate of $p$ and shifting the $i$th through $n$th coordinates of $p$ to the right.  

There are two types of 1-dimensional compression which we will utilize.  
\begin{Definition}
We say that a set $S \subset \Z^k \times \N^{d}$ is \emph{centrally compressed} in the $i$-th coordinate ($1 \leq i \leq k$) with respect to $p \in \Z^{k-1}\times \N^{d}$ if the set
\begin{equation*}
\{x \in \Z: (p, x \rightarrow i) \in S\}
\end{equation*}
is either empty or of one of the following two forms:
\begin{align*}
\{x: -a  \leq x &\leq a \text{ for } a \in \N\} \\
& \text{OR} \\
\{x: -a \leq x &\leq a+1 \text{ for } a \in \N\}
\end{align*}

\end{Definition}

\begin{Definition}
We say that $S \subset \Z^k \times \N^{d}$ is \emph{downward compressed} in the $j$-th coordinate $(k+1 \leq  j \leq k+d)$ with respect to $p \in \Z^{k} \times \N^{d-1}$ if the set
\begin{equation*}
\{x \in \N: (p, x \rightarrow j) \in S\}
\end{equation*}
is either empty or of the form:
\begin{equation*}
\{x: 0 \leq x \leq a \text{ for } a \in \N\}
\end{equation*}
\end{Definition}

For certain subsets $S \subset \Z^k \times \N^d$, we will calculate the boundary 
\begin{align*}
\partial S &= \{x \in \Z^k \times \N^d: \|y-x\|_\infty \leq 1 \text{ for some } y \in S\} \\
&= \left\{x \in \Z^k \times \N^{d}:\exists \; s \in S \text{ such that }x = s+\epsilon \text{ for some } \epsilon \in \{-1,0,1\}^{k+d}\right\}
\end{align*}
Specifically, we will calculate the boundary of sets which are centrally compressed in the $i$th coordinate with respect to any $p \in \Z^{k-1} \times \N^{d}$   for $1 \leq i \leq k$, and which are downward compressed in the $j$th coordinates with respect to any $p \in  \Z^k \times \N^{d-1}$ for $k+1 \leq j \leq k+ d$.  We will say that such sets are ``compressed in every coordinate.''  Before completing the calculations, we shall see that a set of minimum boundary must be compressed in every coordinate.

\subsection{Effect of Compression}

\begin{Definition}
Let $S \subset \Z^k \times \N^{d}$.  For $1 \leq i \leq k$, we define $S_i$ to be the $i$th central compression of $S$  by specifying its 1-dimensional sections in the $i$th coordinate.  Specifically, 
\begin{enumerate}
\item For each $p \in \Z^{k-1} \times \N^{d}$, 
\begin{equation*}
|\{x \in \Z: (p, x \rightarrow i) \in S_i\}| 
= |\{x \in \Z: (p, x \rightarrow i) \in S\}| 
\end{equation*}
\item $S_i$ is centrally compressed in the $i$th coordinate with respect to $p$ for each $p \in \Z^{k-1} \times \N^{d}$.
\end{enumerate}
For $k+1 \leq j \leq k+d$, we define $S_j$ to be the $j$th central compression of $S$  by specifying its 1-dimensional sections in the $j$th coordinate.  Specifically, 
\begin{enumerate}
\item For each $p\in \Z^{k} \times \N^{d-1}$, 
\begin{equation*}
|\{x \in \N: (p, x \rightarrow j) \in S_j\}| 
= |\{x \in \N: (p, x \rightarrow j) \in S\}| 
\end{equation*}
\item $S_j$ is downward compressed in the $j$th coordinate with respect to $p$ for each $p \in \Z^{k} \times \N^{d-1}$.
\end{enumerate}

\end{Definition}

In words: after fixing a coordinate $i \in \{1, 2, \dots, k\}$, we consider all lines in $\Z^k \times \N^d$ where only the $i$th coordinate varies, and we intersect those lines with $S$.  Eeach of the points in those intersections are moved along the line so that they are a segment centered around 0.  The result is $S_i$.  Similarly, after choosing $j \in \{k+1, \dots, k+d\}$, we again consider lines in $\Z^k \times \N^d$ where the $j$th coordinate varies, and intersect those lines with $S$.  We now move the points along the line so that they are a segment starting from 0.  The result is $S_j$.

\begin{Proposition}\label{Compression1}
Suppose that $S \subset \Z^k \times \N^{d}$.    For $1 \leq i \leq k$, let $S_i$ be the $i$th central compression of $S$.  Then 
\begin{equation*}
\partial S_i \leq \partial S
\end{equation*}
Similarly, for $k+1 \leq j \leq k+d$, let $S_j$ be the $j$th downward compression of $S$.  Then
\begin{equation*}
\partial S_j \leq \partial S
\end{equation*}

\end{Proposition}

\begin{proof}
Let $1 \leq i \leq k$ and fix $p \in \Z^{k-1} \times \N^{d}$.  Consider any $\epsilon \in \{-1,0,1\}^{k+d-1}$ for which
\begin{equation*}
p+\epsilon \in \Z^{k-1} \times \N^{d}.
\end{equation*}  
Let
\begin{equation*}
I = \{\epsilon \in \{-1,0,1\}^{k+d-1}: \exists \; x \in \Z: (p+\epsilon, x \rightarrow i)  \in S\}.
\end{equation*}
Note that, by the definition of $S_i$, for $p \in \Z^{k-1} \times \N^d$, we have
\begin{equation*}
\{x \in \Z: (p+\epsilon, x \rightarrow i) \in S\} \not= \emptyset
\end{equation*}
if and only if
\begin{equation*}
\{x \in \Z: (p+\epsilon, x \rightarrow i) \in S_i\} \not= \emptyset
\end{equation*}
We have
\begin{multline*}
\{(p, x \rightarrow i):  x \in \Z\}\cap \partial S \\
= \bigcup_{\epsilon \in I}\left\{(p,x \rightarrow i): \exists \; y \in \{x-1, x, x+1\}    \text{such that } (p+\epsilon, y \rightarrow i) \in S\right\}
\end{multline*}
Thus,  we have
\begin{align*}
&|\{(p, x \rightarrow i):  x \in \Z\}\cap \partial S |\\
&\geq \max_{\epsilon \in I}|\left\{(p,x \rightarrow i): \exists \; y \in \{x-1, x, x+1\}    \text{such that } (p+\epsilon, y \rightarrow i) \in S\right\}| \\
&\geq \max_{\epsilon \in I} |\{ x \in \Z: (p+\epsilon, x \rightarrow i) \in S\}| + 2
\end{align*}

where the last inequality is achieved if and only if, for the $\epsilon$ achieving the maximum,
\begin{equation*}
\{x \in \Z: (p + \epsilon, x \rightarrow i) \in S\} = \{x \in \Z: a \leq x \leq b\}
\end{equation*}
for some $a, b \in \Z$.

Additionally, we note that
\begin{multline*}
\{(p, x \rightarrow i):  x \in \Z\}\cap \partial S_i \\
= \bigcup_{\epsilon \in I}\left\{(p, x \rightarrow i): \exists \;y\in \{x-1,x, x+1\} \text{ such that } (p+\epsilon, y \rightarrow i) \in S_i \right\}. 
\end{multline*}
so that now, by the definition of $S_i$, 
\begin{align*}
&|\{(p, x \rightarrow i):  x \in \Z\}\cap \partial S_i  |\\
&= \max_{\epsilon \in I}|\left\{(p, x \rightarrow i): \exists \;y\in \{x-1,x, x+1\} \text{ such that } (p+\epsilon, y \rightarrow i) \in S_i \right\}| \\
&= \max_{\epsilon \in I} |\{ x \in \Z: (p+\epsilon, x \rightarrow i) \in S_i\}| + 2
\end{align*}
By the definition of $S_i$, for each $\epsilon \in I$, we have
\begin{equation*}
 |\{ x \in \Z: (p+\epsilon, x \rightarrow i) \in S\}|  = |\{ x \in \Z: (p+\epsilon, x \rightarrow i) \in S_i\}|.
\end{equation*}
Finally, we note that both $\partial S$ and $\partial S_i$ are the disjoint unions of those 1-dimensional sections:
\begin{align*}
\partial S &= \bigcup_{p \in \Z^{k-1}\times \N^{d}} \{(p, x \rightarrow i):  x \in \Z\}\cap \partial S \\
\partial S_i &= \bigcup_{p \in \Z^{k-1}\times \N^{d}}\{(p, x \rightarrow i):  x \in \Z\}\cap \partial S_i
\end{align*}
Thus, we can conclude that
\begin{equation*}
\partial S \geq \partial S_i.
\end{equation*}

The proof in the case of downward compression in coordinate $j \in \{k+1, \dots, k+d\}$  is similar.
\end{proof}

\subsection{Boundaries of Compressed Subsets of $\Z^k \times \N^{d}$}\label{BoundaryZN}

We have the following definitions:  For fixed $k, d \in \N$ (not both 0), let
\begin{align*}
K &= \{1, 2, \dots, k\} \\
D &= \{k+1, k+2, \dots, k+d\}
\end{align*}

In addition, for $I \subset \{1, 2, \dots, k+d\}$ and $S \subset \Z^k \times \N^d$, let $P_I(S)$ denote the $k+d-|I|$-dimensional projection of $S$ which results from deletion of the coordinates in $I$.

\begin{Theorem}\label{ZNBoundary}
Let $S \subset \Z^k \times \N^{d}$ be a finite set which is centrally compressed in coordinate $i$ for each $1 \leq i \leq k$ and downward compressed in each coordinate $j$ where $k+1\leq j \leq k+d$.  Then 
\begin{equation*}
\partial S = \sum_{I \subset \{1, 2, \dots, k+d\}} 2^{|I\cap K|}|P_I(S)|.
\end{equation*}

\end{Theorem}

\begin{proof}
We proceed by induction on $|S|$.  If $|S| =1$, then $S = \{(0,0,0,\dots,0)\}$ and 
\begin{multline*}
\partial S = \{(x_1,x_2, \dots, x_{k+d}): x_i \in \{-1,0,1\}  \text{ for } 1 \leq i \leq k, \\
x_j \in \{0,1\} \text{ for } k+1\leq j \leq k+d\}
\end{multline*}
so that
\begin{equation*}
|\partial S| = 3^k2^{d}
\end{equation*}
We note that in this case, $|P_I(S)| = 1$ for any $I \subset \{1, 2, \dots, k+d\}$ so that
\begin{align*}
\sum_{I \subset \{1, 2, \dots, k+d\}} 2^{|I\cap K|}|P_I(S)| &= \sum_{I \subset \{1, 2, \dots, k+d\}} 2^{|I\cap K|} \\
&=  \sum_{i=0}^{d}\sum_{j=0}^k {k \choose j}{d \choose i} 2^j \\
&= 3^k2^{d}
\end{align*}

Now suppose that $|S|>1$.  We consider the same well-ordering $\prec$ on $\Z$ as considered in Section \ref{SectionMyTheorem}:
\begin{equation*}
0 \prec 1  \prec  -1  \prec  2  \prec  -2  \prec  3 \cdots
\end{equation*}
and we denote by $<$ the regular ordering on either $\Z$ or $\N$.  As in Section \ref{SectionMyTheorem}, for $a \in \Z$,  we let $a^+$ denote the element of $\{a, a+1, a-1\}$ which is largest in the well-ordering $\prec$ and $a^-$ the element of $\{a, a+1, a-1\}$ which is the smallest.  

Let $z = (z_1, z_2, \dots, z_{k+d}) \in S$ be a point which is a ``corner point'' of $S$.  That is, for any $i \in \{1, 2, \dots, k\}$, the point
\begin{equation*}
(z_1, \dots, z_{i-1}, z_i^+, z_{i+1}, \dots, z_{k+d}) 
\end{equation*}
is not in $S$ and for any $j \in \{k+1, k_2, \dots, k+d\}$, the point
\begin{equation*}
(z_1, \dots, z_{j-1}, z_j+1, z_{j+1}, \dots, z_{k+d}) 
\end{equation*}
is not in $S$.  We note that $z$ exists because $S$ is finite.  Also note that, since $S$ is compressed in every coordinate, for any $i \in \{1, 2, \dots, k\}$ where $z_i \not= 0$, 
\begin{equation*}
(z_1, \dots, z_{i-1}, z_i^-, z_{i+1}, \dots, z_{k+d})  \in S
\end{equation*}
and 
\begin{equation*}
(z_1, \dots, z_{i-1}, z_i^-, z_{i+1}, \dots, z_{k+d})   \not= (z_1, \dots, z_{i-1}, z_i, z_{i+1}, \dots, z_{k+d})  
\end{equation*}
Additionally, for any $j \in \{k+1, k+2, \dots, d+k\}$ where $z_j \not= 0$, 
\begin{equation*}
(z_1, \dots, z_{j-1}, z_j-1, z_{j+1}, \dots, z_{k+d})  \in S
\end{equation*}
and 
\begin{equation*}
(z_1, \dots, z_{j-1}, z_j-1, z_{j+1}, \dots, z_{k+d})   \not= (z_1, \dots, z_{j-1}, z_j, z_{j+1}, \dots, z_{k+d})  
\end{equation*}

By induction,
\begin{equation*}
\partial \left(S\backslash\{z\}\right) = \sum_{I \subset \{1, 2, \dots, k+d\}} 2^{|I\cap K|} |P_I(S\backslash\{z\})|.
\end{equation*}

Define
\begin{equation*}
Z_K = \{i \in \{1, 2, \dots, k\}: z_i = 0\}
\end{equation*}
and
\begin{equation*}
Z_D= \{j \in \{k+1, k+2, \dots, k+d\}: z_j = 0\}
\end{equation*}
Then we note that the only points $(y_1, y_2, \dots, y_{k+d})$ which are in $\partial S$ but not in $\partial \left(S\backslash\{z\}\right)$ are of the form
\begin{equation*}
y_i = \begin{cases}
-1,0, \text{ or } 1 & \text{ if } i \in Z_K \\
y_i^+& \text{ if } i \in \{1, 2, \dots, k\} \backslash Z_K \\
0 \text{ or } 1 & \text{ if } i \in Z_D \\
y_i+1 & \text{ if } i \in \{k+1, k+2, \dots, k+d\} \backslash Z_D \\
\end{cases}
\end{equation*}
Thus, we can see that there are a total of 
\begin{equation*}
3^{|Z_K|}2^{|Z_D|}
\end{equation*}
points in $\partial S$ which are not in $\partial \left(S\backslash\{z\}\right)$ so that
\begin{equation*}
|\partial S| = |\partial S \backslash\{z\}|+3^{|Z_K|}2^{|Z_D|}
\end{equation*}

We also note that since $S$ is compressed in every coordinate, the only $I \subset \{1, 2, \dots, k+d\}$ for which $|P_I(S)|$ is larger than $|P_I(S \backslash\{z\})|$ must be of the form $I \subset Z_K \cup Z_D$.  If $I\subset Z_K \cup Z_D$, then $|P_I(S)| = |P_I(S \backslash\{z\})|+1$ and if $I \not\subset Z_K \cup Z_D$, then $|P_I(S)| = |P_I(S \backslash\{z\})|$.  Thus, we see that
\begin{align*}
 \sum_{I \subset \{1, 2, \dots, k+d\}} 2^{|I\cap K|} |P_I(S)| 
 &=  \sum_{I \subset \{1, 2, \dots, k+d\}} 2^{|I\cap K|} |P_I(S\backslash\{z\})|+ \sum_{I \subset Z_K \cup Z_D} 2^{|I\backslash Z_D|} \\
 &=  \partial \left(S\backslash\{z\}\right)+ \sum_{j=0}^{|Z_K|}\sum_{i=0}^{|Z_D|}{|Z_D| \choose i}{|Z_K| \choose j}2^{j} \\
 &=\partial \left(S\backslash\{z\}\right)+3^{|Z_K|}2^{|Z_D|}
 \end{align*}
 and hence our Theorem is proven.

\end{proof}

We note that, in light of Theorem \ref{MyTheorem} and Theorem \ref{ZNBoundary}, we have the following Corollary:

\begin{Corollary}
Let $X_n$ denote the set of all subsets $X \subset \Z^k$ of size $n$.  Define the function $f:X_n \to \R$ as follows: for $X \in X_n$,
\begin{equation*}
f(X) = \sum_{I \subset \{1, 2, \dots, k\}} 2^{|I|}|P_I(X)|.
\end{equation*}
The minimum value of this function is achieved at an initial segment of $\Z^k$ of size $n$, according to the well-ordering $\prec$

\end{Corollary}

Additionally, in light of Corollary \ref{TheoremN} and Theorem \ref{ZNBoundary}, we have the following Corollary:

\begin{Corollary}
Let $Y_n$ denote the set of all subsets $Y \subset \N^k$ of size $n$.  Define the function $f:Y_n \to \R$ as follows: for $Y \in Y_n$,
\begin{equation*}
f(Y) = \sum_{I \subset \{1, 2, \dots, k\}} |P_I(Y)|.
\end{equation*}
The minimum value of this function is achieved at an initial segment of $\N^k$ of size $n$, according to the well-ordering $\prec_\N$

\end{Corollary}

\section{Final Remarks}\label{Final Remarks}

The calculations in Section \ref{Boundary Computations} allow us to understand some of the behavior of the boundaries of sets in our graph.  These calculations rely on the technique which we call centralization.  While compression is a more powerful technique, it requires that sets of minimal boundary be nested.  The process of centralizing and computing boundaries does not require that sets of minimal boundary have any particular form.  Perhaps these ideas can be used in other situations where the sets of minimal boundary are more complicated.

\bibliographystyle{plain}
\bibliography{VertexIsoperimetry}
\end{document}